\documentclass[12pt]{article}
\usepackage{epsf}
\usepackage{amsbsy,amsmath}
\usepackage{mathtools}
\usepackage{mathrsfs}
\usepackage{amsfonts}
\usepackage{amssymb}
\usepackage{enumitem}
\usepackage{eucal}
\usepackage{enumitem}
\usepackage{graphics,mathrsfs}
\usepackage{amsthm}
\usepackage{secdot}
\usepackage{xcolor}
	
\newtheorem{theorem}{Theorem}[section]
\newtheorem{lemma}[theorem]{Lemma}
\newtheorem{proposition}[theorem]{Proposition}

\theoremstyle{definition}
\newtheorem{definition}[theorem]{Definition}
\newtheorem{example}[theorem]{Example}

\newtheorem{conjecture}[theorem]{Conjecture}

\theoremstyle{remark}
\newtheorem{remark}[theorem]{Remark}
\numberwithin{equation}{section}

\numberwithin{equation}{section}

\newsavebox{\savepar}

\begin{document}
		\title{\sc A nonlocal elliptic problem \\ on a Heisenberg group} 
		
			\author{\sc Debajyoti Choudhuri$^{a}$
			\footnote{
			Corresponding author: 
		dchoudhuri@iitbbs.ac.in
			}, Lamine Mbarki$^{b}$, \\\& {\sc Olimpio Hiroshi Miyagaki$^{c}$}\\
			\small{$^{a}$School of Basic Sciences, Indian Institute of Technology Bhubaneswar,}\\ [-0.2cm]
			\small{ Khordha, 752050, Odisha, India}\\ [-0.2cm]
					\small{$^{b}$Mathematics Department, Faculty of Science of Tunis,}\\ [-0.2cm]
					\small{University of Tunis El Manar, 2092, Tunis, Tunisia}\\[-0.2cm]
					\small{$^{c}$ Departamento de Matem\~{a}tica, Universidade Federal de S\~{a}o Carlos,}\\[-0.2cm] \small{CEP:13565-905, S\~{a}o Carlos-SP, Brazil }
		}
		

		\date{}
		\maketitle

		
		
		\maketitle	
		\begin{abstract}
			\noindent We study an elliptic nonlocal problem driven by a source term which is a Radon measure. 
			We prove the existence of a weak solution by a weak convergence method. 
			In the process, we define a new function space which we name  the {\it Walker space}. 
			The question about the  existence of infinitely many nontrivial solutions leads to an interesting conjecture.
			\begin{flushleft}
				{\sl Math. Subj. Classif. (2020)}:~35J20, 35J35, 35J60, 46E35.
			\end{flushleft}
		\end{abstract}
		
		\section{Introduction}\label{s1}
		
		The model problem we study in this paper is as follows:
		\begin{equation}\label{main}
				\left\{\begin{aligned}
			-\mathfrak{m}\left(\int_{B_{\mathbb{H}}(0,1)}|\nabla_{\mathbb{H}}u|^2dx\right)\mathcal{L}_{\mathbb{H}}u+\mu\vec{U}\cdot\nabla_{\mathbb{H}}u&=\nu~\text{ \textcolor{blue}{in} }
			B_{\mathbb{H}}(0,1)\subset\mathbb{R}^3,\\
				u&=0~\text{ on }
				\partial B_{\mathbb{H}}(0,1),
			\end{aligned}
			\right.
		\end{equation}
		where $\mathfrak{m}:\mathbb{R}^+\to\mathbb{R}^+$ is the Kirchhoff operator, $\mu\geq 0$, $dx=dx_1dx_2dx_3$ is the Haar measure, $\mathcal{L}_{\mathbb{H}}$ is the sub-Laplacian operator, 
		$\nu$ is a positive Radon measure,
		$\mathbb{H}$ refers to the {\it Heisenberg group}, $B_{\mathbb{H}}(0,1)$ is a unit ball in $\mathbb{R}^3$ with non-smooth Lipschitz boundary $\partial  B_{\mathbb{H}}(0,1)$,
							$\vec{U}:B_{\mathbb{H}}(0,1)\to\mathbb{R}^3$ is a {\it horizontal vector valued function}. The readers may note that although the problem has been studied for dimension $3$ owing to its practical applicability, the proofs however works for dimensions greater than $3$ as well.
							
							The Hausdorff dimension, denoted by $Q$, of the Heisenberg group $(\mathbb{R}^3,\cdot)$ is $4$
				 and	
		 the group operation `$\circ$' on $\mathbb{H}$ is defined as follows:
		  \begin{align}\label{gp_oper}
		 (x_1,x_2,x_3)\circ (y_1,y_2,y_3)=(x_1+y_1,x_2,y_2,x_3+y_3+2^{-1}(x_1y_2-x_2y_1)).
		 \end{align}

		   The following will be assumed:
		\begin{enumerate}[label=(${A}_{{\arabic*}}$)]
			\item   \label{A1} There exist $0<A\leq B$ and $0<\alpha<1$ such that 
			$$At^{\alpha}\leq\mathfrak{m}(t)\leq Bt^{\alpha}, \hbox{ for every } t\in\mathbb{R}^+.$$
		\end{enumerate}	
		
			We impose condition \ref{A1} as a generalized Kirchhoff-type growth condition. Note that this is not of the standard form $m(t)=a+bt$, but still captures nonlocal tension effects. 
		The Kirchhoff operator $\mathfrak{m}$ is said to be {\it degenerate} if $\mathfrak{m}(0)=0$, otherwise it is said to be nondegenerate.	
		
		With the passage of time many researchers have worked with this operator. Fiscella-Valdinoci \cite{FV1} proved the existence of nonnegative solutions of a Kirchhoff type problem driven by a nonlocal integrodifferential operator and a power nonlinearity of critical order. We refer to the works of Bhowmick-Ghosh \cite{SS1}, Choudhuri \cite {chou1}, Zuo et al. \cite{zuo1},
		 and the references therein for more information on the trends and techniques adopted for handling elliptic problems involving the Kirchhoff operator. 
		 		
		However, to the best of  our knowledge there are no publications in which the authors  consider the convective term $\vec{U}\cdot\nabla u$ or a measure or both. We shall use the {\it weak convergence method} -  a common practice when dealing with measure data problems. We refer the readers to the work of Bhakta-Marcus \cite{BM1}, Br\'{e}zis \cite{brezis1}, Giri-Choudhuri \cite{GC1, GC2}, Mingione \cite{mingione1} and the references therein. 
		
		We would also like to bring to the attention of the reader that in all the above mentioned articles, the problems were considered in the Euclidean setup and not in the noncommutative setup. In this article we consider the problem on the Heisenberg group which is a new addition to the existing literature. For more details on the Heisenberg group, we refer to Bonfiglioli et al. \cite{BLU}, Ghosh et al. \cite{SVM1}, Choudhuri-Repov\v{s} \cite{CR1}, and the references therein.
		
	The measure data problems of the type \eqref{main} have not been studied on the horizontal Sobolev spaces over a Heisenberg group. The motivation behind studying  problem \eqref{main} can be drawn from the fact that when $\mathfrak{m}(\cdot)=1,$ we obtain the usual convection-diffusion type equation driven by a source.
		The presence of the convective term indicates that we cannot work with an energy functional. 
			\begin{remark}\label{note1}
	Henceforth, we shall denote the unit ball $B_{\mathbb{H}}(0,1)$ by $\Omega$. 
	\end{remark}
	
		
			The following are the key novelties of this work:
		\begin{itemize}\label{novelties}
		\item Convection and a measure data driven elliptic problem on a Heisenberg group have been studied which, to our knowledge, is new in the literature.
		\item We define a new type of space which we call
		the {\it Walker space}. Without the use of this space one can, at best, prove only
		 that there exists a weak subsolution.
		\item The results in this paper are rigorously proved for dimension $3$. However, the proof strategy also works for dimension $\ge 4$.
		\end{itemize}
		
		

		The paper is organized as follows. In 
		Section \ref{s2}
		we recall the preliminaries needed to study  problem \eqref{main}. In 
		Section \ref{s3} 
		we prove key auxiliary results. In 
		Section \ref{s4}
		we prove the main result of this paper. In 
		Section \ref{s5}
		we  discuss our main result and give an illustrative example.		
		
		\section{Preliminaries}\label{s2}
		We begin by reviewing the fundamentals of Heisenberg group $\mathbb{G}$. The Heisenberg group $\mathbb{H}$ is a well-known example of such a group. The elements of the associated Lie algebra \(\mathfrak{h}\) of $\mathbb{H}$ can  be identified by the left-invariant vector fields which are defined  at a point \((x, y, z)\) by
		\[
		X_1 := \partial_x - \frac{y}{2}\partial_z, \quad X_2 := \partial_y + \frac{x}{2}\partial_z, \quad X_3 := \partial_z.
		\]
		
		Since the nilpotency is of step $2$, we have $N_1 = 2$, due to which the horizontal layer of $\mathbb{H}$ is spanned by two vector fields $X_1,X_2$ with the vertical direction $T = \partial_z$. In other words $V_1 = \text{span}\{X_1, X_2\}$.
		
		 The Heisenberg Laplacian (or sub-Laplacian) is expressed as follows
		$$\mathcal{L}_{\mathbb{H}}:=X_1^2+X_2^2.$$		
		Using the vector fields $X_1,X_2,X_3$, the bounded, horizontal, divergence-free vector-valued function $\vec{U}$ in \eqref{main} at a point $(x,y,z)\in\mathbb{R}^3$ is defined as
		$$\vec{U}(x,y,z)=\alpha(x,y,z)X_1|_{(x,y,z)}+\beta(x,y,z)X_2|_{(x,y,z)}.$$

 Next, we recall the mathematical tools needed to prove the main results. For $1\leq p<\infty$, the Sobolev space is defined as follows
		$$W_0^{1,p}(\Omega):=\{u\in L^p(\Omega):|\nabla_{\mathbb{H}} u|\in L^p(\Omega),u|_{\partial\Omega}=0\}$$
		and is equipped with the following norm 
		$$\|u\|_{1,\textcolor{blue}{p}}:=\left(\int_{\Omega}|\nabla_{\mathbb{H}} u|^pdx\right)^{1/p}.$$
		
		\begin{definition}
		Define a new function space, called the{\it  Walker space,} as follows:
		$$S_M(\Omega):=\{u\in W_0^{1,2}(\Omega)\cap L^{\infty}(\Omega):\mathcal{N}(|\nabla_{\mathbb{H}}u|^2)\leq M\},$$
		where $\mathcal{N}(\cdot)$ refers to the number of modes (i.e., connected nodal regions corresponding to local maxima/minima) of the function $u$. This precise notion formalizes the idea of `peaks'.
		\end{definition}
		
		\begin{remark}\label{walker_space}
			We call  $\mathcal{X}:=S_M(\Omega)\subset W_0^{1,2}(\Omega)$ the {\it Walker space}, owing to the result due to Walker \cite[Theorem $4.2 $]{walker1}. The space $\mathcal{X}$ is the admissible class of solution space equipped with the $W_0^{1,2}$-norm topology.
			\end{remark}
			
		This space $\mathcal{X}$ is nonempty since
		   the space		   
		   $C_{c}^{\infty}(\Omega)$
		    of compactly supported smooth functions with at the most $M$ modes is a subspace of $\mathcal{X}$. The motivation behind assumption on the functions to have finite modes is to avoid infinite number of oscillations.


		We now recall the definition of the convergence in the space of  Radon measures $\mathcal{M}(\Omega)$.
		\begin{definition}\label{meas_conv}
		We say that a sequence $(\nu_n)$ {\it converges to $\nu$  in the sense of measure} in $\mathcal{M}(\Omega)$, and denote it as $\nu_n\rightharpoonup \nu,$  provided that
		$$\int_{\Omega}\varphi d\nu_n\to\int_{\Omega}\varphi d\nu,
		\hbox{ for every }
		\varphi\in C_0(\bar{\Omega}),$$
		where
		$$C_0(\bar{\Omega}):=\{\varphi:\Omega\to\mathbb{R}:\varphi~\text{is continuous and}~\varphi|_{\partial\Omega}=0\}.$$
		 
		\end{definition}
		\begin{remark}\label{density1}
		The space $C_{c}^{\infty}(\Omega)$ is dense in $\mathcal{M}(\Omega)$.
		\end{remark}
		
		 We further recall the space of functions of bounded variation which is defined as follows.
		\begin{definition}\label{BV_space}
			A function $u\in L^1(\Omega)$ is said to be of {\it bounded variation} if there exists a finite vector Radon measure $|\nabla_{\mathbb{H}}u|\in\mathcal{M}(\Omega)$ such that 
			\begin{align}\label{cond_BV}
				\int_{\Omega}u \text{div}_{\mathbb{H}}\vec{\phi} dx=&-\int_{\Omega}\vec{\phi}\cdot\nabla_{\mathbb{H}}udx~\forall~\vec{\phi}\in C_c^1(\Omega,\mathbb{R}^3).
			\end{align}
			Equivalently, if $$V(u,\Omega):=\sup\{\int_{\Omega}u \text{div}_{\mathbb{H}}\vec{\phi} dx:\vec{\phi}\in C_c^1(\Omega,\mathbb{R}^3), \|\vec{\phi}\|_{\infty}\leq 1\}<\infty$$
			then $u$ is said to be of bounded variation and the space of such functions is denoted by $\text{BV}(\Omega)$.
		\end{definition}
		In fact $u$ is considered as a linear functional over $C_c^1(\Omega,\mathbb{R}^3)$. Henceforth we will use the following notation $$V(u,\Omega)=\int_{\Omega}|\nabla_{\mathbb{H}}u|dx.$$
		\begin{remark}\cite[Theorem $5.3.4$]{ziemer1}\label{BV_Compact}
			The space $\text{BV}(\Omega)$ is complete with respect to the norm $\int_{\Omega}|\nabla_{\mathbb{H}}u|dx$ and is compactly embedded in $L^1(\Omega)$.
		\end{remark}
		
		\begin{remark}\label{constant}
		Throughout this manuscript the constants will be denoted by $C$.
		\end{remark}

		Next, we recall the notion of the  Marcinkiewicz space (or the weak $L^p$ space) which will be denoted by $\mathcal{M}^p(\Omega)$, $0<p<\infty,$ and is defined for every measurable function $f:\Omega\to\mathbb{R}$. 
		The corresponding distribution function obeys the estimate 
		$$\mathcal{T}_f(t):=m(\{x\in\Omega:|f(x)|>t\}\leq\frac{C}{t^q},
		\hbox{ for every } t\geq t_0,$$ 
		where $C, t_0>0$ and $m$ denotes the Lebesgue measure. For more information on the Marcinkiewicz estimates the readers may refer to Mingione \cite{mingione2}.
		
		The presence of  very rough data $\nu$ motivates us to us to investigate the existence of {\it duality solutions}. 

			For clarity. we recall the following necessary definitions.
		
		\begin{definition}\label{weak_soln}
			We shall call $u\in \mathcal{X}$ a {\it weak solution} of problem \eqref{main} if 
			\begin{align}\label{defn_WS}
				\begin{split}
					\mathfrak{m}(\|u\|_{1,2}^2)\int_{\Omega}\nabla_{\mathbb{H}}u\cdot \nabla_{\mathbb{H}}\varphi dx&+\mu\int_{\Omega}\varphi\vec{U}\cdot\nabla_{\mathbb{H}}udx=\int_{\Omega}\varphi d\nu,~
					\hbox{ for every }
					\varphi\in C_{c}^{\infty}(\Omega).
				\end{split}
			\end{align}
		\end{definition}
			Here $\|u\|_{1,2}^2=\int_{\Omega}|\nabla_{\mathbb{H}}u|^2dx$ denotes the standard $W_0^{1,2}$-norm.
		\begin{definition}\label{duality_soln}
			We shall call $u \in L^1(\mathbb{R}^3)$ a {\it duality solution} of problem \eqref{main}		
			\begin{align}\label{defn1}
				\int_{\mathbb{R}^3}ugdx&=	
				\int_{\mathbb{R}^3}vd\mu,~\text{ for every}~g\in C_{c}^{\infty}(\Omega)
			\end{align}
			where $v$ is a solution of the following problem:		
			
			\begin{equation}
				\left\{
				\begin{aligned}\label{meas_prob2}				
					-\mathfrak{m}(\|u\|_{1,2}^2)\mathcal{L}_{\mathbb{H}}v+\mu\vec{U}\cdot\nabla_{\mathbb{H}}u
					&= g~\text{\textcolor{blue}{in}}~\Omega\\
					v&=0~\text{on}~\partial\Omega.
				\end{aligned}			
				\right.
			\end{equation}
		\end{definition}

		The main result of this paper is as follows:
		\begin{theorem}\label{existence}
			Problem \eqref{main} has at least one weak solution. 
		\end{theorem}	
			
		\begin{remark}\label{reg1}
		The solution $v$ of problem \eqref{meas_prob2} is in $C_{\text{loc}}^{1,\gamma}\cap C_{c}(\bar{\Omega})$ by a combination of the arguments by 
		Choudhuri  \cite[Theorem $3.8$]{chou1}
		 and 
		  Su et al.  \cite[Theorem 1.3, 1.4]{vald1}. 
		  The low regularity of $v$ as compared to the one obtained in 
		  \cite[Remark $1.7$]{vald1}
		   is due to the boundary $\partial B_{\mathbb{H}}(0,1)$ being Lipschitz continuous and not $C^{1,1}$-regular.
		\end{remark}	
			
		For other notions, results and tehniques, needed in this paper, the reader can refer to Papageorgiou et al.  \cite{PRR}.
		
		\section{Auxiliary Results}\label{s3}
		For each $n\in\mathbb{N}$ we shall define our problem as follows:
		\begin{equation}\label{Pn}
		\left\{\begin{aligned}
				-\mathfrak{m}\left(\int_{B_{\mathbb{H}}(0,1)}|\nabla_{\mathbb{H}}u|^2dx\right)\mathcal{L}_{\mathbb{H}}u+\mu\vec{U}\cdot\nabla_{\mathbb{H}}u_n&=\nu_n~\text{ \textcolor{blue}{in}}~B_{\mathbb{H}}(0,1)\subset\mathbb{R}^3,\\
			u&=0~\text{ on}~\partial B_{\mathbb{H}}(0,1),
		\end{aligned}
		\right.
	\end{equation}
	where $(\nu_n)$ is a sequence of functions from $C_c^{\infty}(\Omega)$ such that $\nu_n\rightharpoonup \nu,$ as $n\to\infty$.
				Any solution of problem \eqref{Pn} will be denoted by
				 $u_n,$ thereby generating a sequence of solutions $(u_n)$. Thus the solution $u_n$ satisfies the weak formulation of
				 problem 
				  \eqref{Pn} which is as follows
				\begin{align}\label{WFn}
					\begin{split}
					\int_{\Omega}\mathfrak{m}(\|u_n\|_{1,2}^2)\nabla_{\mathbb{H}}u_n\cdot \nabla_{\mathbb{H}}\phi dx&+\mu\int_{\Omega}\phi\vec{U}\cdot\nabla_{\mathbb{H}}u_ndx=\int_{\Omega}\phi d\nu_n~\hbox{ for every }
					\phi\in C_{c}^{\infty}(\Omega).
					\end{split}
				\end{align}
			
				\begin{lemma}\label{completenss}
				The space $\mathcal{X}$ is both complete and reflexive with respect to the norm $\|\cdot\|_{1,2}$.
				\end{lemma}
				\begin{proof}
				Completeness follows from the definition of the norm and the fact that $\mathcal{N}(|\nabla_{\mathbb{H}}(\cdot)|^2)\leq M$ for each $v\in \mathcal{X}$. This is because, if $(u_n)$ is a Cauchy sequence in $\mathcal{X}$, then there exists $u\in W_0^{1,2}(\Omega)\cap L^{\infty}(\Omega)$ such that $u_n\to u,$ as $n\to\infty$ in $W_0^{1,2}(\Omega)$. 								
				Since $\mathcal{N}(|\nabla_{\mathbb{H}}u_n|^2)\leq M,$  for every $n\in\mathbb{N}$, 
				it follows that
				 $\mathcal{N}(|\nabla_{\mathbb{H}}u|^2)\leq\underset{n\to\infty}\liminf~\mathcal{N}(|\nabla_{\mathbb{H}}u_n|^2)\leq M,$
				  therefore $u\in \mathcal{X}$. 						  	  
				Clearly, $\mathcal{X}$ is a closed subspace of $W_0^{1,2}(\Omega),$
				hence $\mathcal{X}$ is reflexive as well. 
				
				This completes the proof of Lemma \ref{completenss}.
				\end{proof}
				
				We now prove the following result for the approximating problem.
				
				\begin{lemma}\label{Pn_existence}
				For each $\nu_n$, there exists a solution $u_n$ of problem \eqref{Pn}. 
				\end{lemma} 
				\begin{proof}
			We begin by considering the bounded linear functional
			$$P_n(v):=\int_{\Omega}v\nu_ndx.$$
			 Let $B:W_0^{1,2}(\Omega)\to (W_0^{1,2}(\Omega))^*:=W^{-1,2}(\Omega)$ be the map $u\mapsto B_u$.  To each $u\in S_M(\Omega)\subset W_0^{1,2}(\Omega)$ let us define
			$$B_u(v):=\frac{1}{2}\int_{\Omega}\mathfrak{m}\left(\int_{\Omega}|\nabla_{\mathbb{H}}u|^2dx\right)\nabla_{\mathbb{H}} u\cdot \nabla_{\mathbb{H}} vdx+\mu\int_{\Omega}v\vec{U}\cdot\nabla_{\mathbb{H}}udx.$$
			Apparently $B_u$ is a member of $W^{-1,2}(\Omega)$ and is a strictly monotone operator (see \cite{GBS1}). By an application of the Minty-Browder theorem (see \cite{RR1}) there exists $u_n$ such that $B(u_n)=P_n$. This completes the proof of Lemma \ref{Pn_existence}.
				\end{proof}
				\begin{remark}\label{nonuniqueness}
				We note that the uniqueness of $u_n$ is not established in the above proof since it is not necessary. However, the uniqueness can be achieved by assuming that the Kirchhoff operator $\mathfrak{m}$ is strictly increasing.
				\end{remark}
				We shall now show that the sequence of the approximating solutions $(u_n)$ is bounded in the Marcinkiewicz space.
				
				\begin{lemma}\label{bdd_Marcin_sp}
				The sequence of solutions $(u_n)$ is bounded in $\mathcal{M}^{{\frac{4+8\alpha}{3+5\alpha}}}(\Omega)$.
				\end{lemma}
				\begin{proof}
				We recall that a truncation of any measurable function $h$ is defined as 
				\[T_kh(x)=\begin{cases}
					h(x), & \text{if}~|h(x)|\leq k\\
					k, & \text{otherwise}.
				\end{cases}\]
				We shall now test the weak formulation \eqref{WFn} with $\phi=T_k(u_n)$ and use condition \ref{A1} to get the following
				\begin{align}\label{eq1}
					\begin{split}
						A\|u_n\|^{2\alpha}\leq&2^{-1}\left(\int_{\Omega}\mathfrak{m}(\|u_n\|_{1,2}^2)\nabla_{\mathbb{H}}u_n\cdot \nabla_{\mathbb{H}}T_k(u_n) dx\right)\\
					=&\mu\int_{\Omega}u_n\vec{U}\cdot\nabla_{\mathbb{H}}u_ndx+\int_{\Omega}T_k(u_n)  d\nu_n\\
					\leq & \mu\int_{\Omega}u_n\vec{U}\cdot\nabla_{\mathbb{H}}u_ndx+k\|\nu_n\|_1
					\leq  C\mu\|u_n\|_{1,2}+Ck.
					\end{split}
				\end{align}
				Thus we have
				\begin{equation}\label{eq2}
					A\|u_n\|_{1,2}^{2\alpha+2}
					\leq 
					 C\mu\|u_n\|_{1,2}+\lambda k^4m(|\Omega|)+Ck
					\leq  
					C\mu\|u_n\|_{1,2}+Ck.
				\end{equation}				
				This implies that the sequence $(u_n)$ is bounded in $W_0^{1,2}(\Omega),$ provided that $\mu<\frac{A}{C}$. 
				
				However, the bound depends on $k$ and hence it is not   $W_0^{1,2}$-uniformly bounded. Without loss of generality, we may assume that this $k$-dependent bound can be of the order of $k^{4}$. As a result, we have 
				\begin{align}\label{bdd1}
				\begin{split}
				\|u_n\|_{1,2}\leq & Ck^\frac{1}{2+2\alpha}.
				\end{split}
				\end{align}
				therefore we have to prove that 
				$$m(\{x:|\nabla_{\mathbb{H}}u_n|\geq t\})\leq\frac{C}{t^{2}},~\hbox{ for every }
				 t\geq 1,$$
				so that it can be shown to be uniformly bounded in $\mathcal{M}^{2}(\Omega)$.  
				
				We begin to prove this boundedness by noting  
				the following set inclusions:
				\begin{align}\label{set_incl1}
				\begin{split}
				\{x:|\nabla_{\mathbb{H}}u_n|\geq t\}&=	\{x:|\nabla_{\mathbb{H}}u_n|\geq t, |u_n|<k\}\cup	\{x:|\nabla_{\mathbb{H}}u_n|\geq t, |u_n|\geq k\}\\
				&\subset \{x:|\nabla_{\mathbb{H}}u_n|\geq t, |u_n|<k\}\cup\{x:| u_n|\geq k\}.
				\end{split}
				\end{align}
				
				By the subadditivity of the Lebesgue measure we have
					\begin{align}\label{set_incl2}
					\begin{split}
						m(\{x:|\nabla_{\mathbb{H}}u_n|\geq t\})&\leq m(	\{x:|\nabla_{\mathbb{H}}u_n|\geq t, |u_n|<k\})+m(\{x:|u_n|\geq k\}).
					\end{split}
				\end{align}
				
				Let $\mathcal{S}_1:=\{x:|u_n|\geq k\}$. By the embedding result of Burns \cite{burns1} and \eqref{bdd1} we have
				\begin{align}\label{bbd2}
				\begin{split}
				k^2m(\mathcal{S}_1)^{1/2}=\left(\int_{\mathcal{S}_1}k^{4}dx\right)^{2/4}\leq &C\int_{\Omega}|\nabla_{\mathbb{H}}u_n|^2dx\leq Ck^{\frac{1}{1+\alpha}},
				\end{split}
			    \end{align}
			    hence
			    \begin{align}\label{bbd3}
			    	\begin{split}
			    		m(\{x:|u_n|\geq k\})&\leq \frac{C}{k^{\frac{2+4\alpha}{1+\alpha}}}.
			    	\end{split}
			    \end{align}
			    
			    Next, consider the following for $t>1$:
			    \begin{align}\label{bbd4}
			    	\begin{split}
			    		m(\{x:|\nabla_{\mathbb{H}}u_n|\geq t, |u_n|< k\})&\leq \frac{1}{t^4}\int_{\Omega}\mathfrak{m}(\|u_n\|_{1,2}^2)|\nabla_{\mathbb{H}}u_n|^2dx\\
			    		&\leq \frac{1}{t^2}\int_{\Omega}\mathfrak{m}(\|u_n\|_{1,2}^2)|\nabla_{\mathbb{H}}u_n|^2dx\\
			    		&\leq\frac{Ck}{t^2}.
			    	\end{split}
			    \end{align}
			    On combining \eqref{bbd3} and \eqref{bbd4}, we obtain
			    \begin{align}\label{bbd5}
			    	\begin{split}
			    		m(\{x:|\nabla_{\mathbb{H}}u_n|\geq t\})&\leq \frac{Ck}{t^2},
			    	\end{split}
			    \end{align}
			    and choosing
			     $k=t^{\frac{2+2\alpha}{3+5\alpha}},$ 
			     we arrive at
			    \begin{align}\label{bbd6}
			    	\begin{split}
			    			m(\{x:|\nabla_{\mathbb{H}}u_n|\geq t\})&\leq \frac{C}{t^{\frac{4+8\alpha}{3+5\alpha}}},
			    	\end{split}
			    \end{align}
			    therefore, $(u_n)$ is indeed bounded in $\mathcal{M}^{{\frac{4+8\alpha}{3+5\alpha}}}(\Omega)$.
			    
			    This completes the proof of Lemma \ref{bdd_Marcin_sp}.
				\end{proof}
				\begin{remark}\label{lim1}
				In the limiting case, one can observe that as $\alpha\to 0^+$, the sequence is bounded in $\mathcal{M}^{4/3}(\Omega),$ which is actually equal to $\frac{Q}{Q-1}$. In the current case, $Q=4$.
				\end{remark}				
				\section{Proof of the Main Result 
				}\label{s4}
				
				We are now in a position to prove the main result stated in Theorem \ref{existence}.				
				We have already assumed that $\nu_n\rightharpoonup \nu$ in $\mathfrak{M}(\Omega)$. Since $L^r(\Omega)\hookrightarrow \mathcal{M}^r(\Omega)\hookrightarrow L^{r-\epsilon}(\Omega)$, we
				can
				 conclude that $(u_n)$ weakly converges to some $u$ in $W_0^{1,r}(\Omega),$ for $1\leq r<\frac{4+8\alpha}{3+5\alpha}$. 
				
				 We now claim that $\mathfrak{m}(\|u_n\|_{1,2}^2)$ converges to a finite number. 
				 For if not, then either (a)~it does not converge or (b)~there exists a subsequence such that $\mathfrak{m}(\|u_n\|_{1,2}^2)\to\infty,$ as $n\to\infty$. 
				
				In the case of (a), by the Bolzano-Weierstrass theorem, there exists a subsequence of $\mathfrak{m}(\|u_n\|_{1,2}^2)$ that converges to a finite number. 
				
				In the case of (b) however, this is a contradiction,  since then
				 the left hand side of  \eqref{defn_WS} yields $\infty$ and the right hand side of it is finite. 
				
				Therefore by \ref{A1}, there exists $C>0$ such that $\|u_n\|_{1,2}\to C^{-1}l^{\frac{1}{2\alpha}}$. This $l$ depends on the choice of $(\nu_n)$. Thus there exists a subsequence of $(u_n)$ such that $u_n\rightharpoonup u$ in $\mathcal{X}$. Thus, as $n\to\infty$ in \eqref{WFn}, we have, by the weak lower semicontinuity of $\|\cdot\|_{1,2}$-norm, the following
				\begin{align}\label{eq1'}
				\begin{split}
				\mathfrak{m}(\|u\|_{1,2}^2)\int_{\Omega}\nabla_{\mathbb{H}}u\cdot\nabla\varphi dx&+\mu\int_{\Omega}\phi\vec{U}\cdot\nabla_{\mathbb{H}}udx\\
				&\leq\int_{\Omega}\varphi d\nu,~\hbox {for every }
				\varphi\geq 0\in C_{c}^{\infty}(\Omega),
				\end{split}
				\end{align}	
				therefore $u$ is a {\it subsolution} to \eqref{main}. 
				
				Moreover, by Proposition \ref{s_to_1} we have $u_n\to u$ in $\mathcal{X}$ since the number of modes are finite ($\leq M$). Thus $u$ is indeed a weak solution of problem \eqref{main}.
				

				\begin{remark}\label{OHM}
				The obtained solutions are all nontrivial since we have considered a nonhomogeneous measure data problem.
				\end{remark}
				
				In order to complete the proof of the main result it now suffices to verify  the following result.
				\begin{proposition}\label{dual_weak}
				The following two statements are equivalent:
				\begin{enumerate}[label=\roman*]
				\item $u$ is a weak solution of problem \eqref{main};
				\item $u$ is a duality solution of problem \eqref{main}.
				\end{enumerate}
				\end{proposition}
				\begin{proof}
				
				Let $u$ be a weak solution of problem \eqref{main}. Then we have 
				\begin{align}\label{defn_WS1}
					\begin{split}
						\int_{\Omega}\mathfrak{m}(\|u\|_{1,2}^2)\nabla_{\mathbb{H}}u\cdot \nabla_{\mathbb{H}}\varphi dx&+\mu\int_{\Omega}\varphi\vec{U}\cdot\nabla_{\mathbb{H}}udx\\
						&=\int_{\Omega}\varphi d\nu,~\hbox{ for every }
						\varphi\in C_{c}^{\infty}(\Omega).
					\end{split}
				\end{align}
				
				Choose any $g\in C_{c}^{\infty}(\Omega)$ and let $v$ be the corresponding solution of problem  \eqref{meas_prob2}. Then by integration by parts, we get 
				\begin{align}\label{int_bp1}
				\begin{split}
				\int_{\Omega}ugdx=
				\int_{\Omega}\mathfrak{m}(\|u\|_{1,2}^2)\nabla_{\mathbb{H}}u\cdot \nabla_{\mathbb{H}}v dx
				&+\mu\int_{\Omega}v\vec{U}\cdot\nabla_{\mathbb{H}}udx\\
				&=\int_{\Omega}v d\nu,
				\end{split}
			    \end{align}
			    which is well-defined by Remark \ref{reg1}. Thus $u$ is also a duality solution of problem \eqref{main}. 
			    
			    Conversely, let $u$ be a duality solution 
			    of problem \eqref{main}
			    and let $\bar{u}$ be a weak solution
			    of problem \eqref{main}. 
			    Then by the above argument, $\bar{u}$ is also a duality solution   of problem \eqref{main}. Hence we have
			    \begin{align}\label{int_bp2}
			    	\begin{split}
			    		\int_{\Omega}ugdx=\int_{\Omega}vd\nu=\int_{\Omega}\bar{u}gdx,
			    		 \hbox{ for every }
			    		 g\in C_{c}^{\infty}(\Omega),
			    		\end{split}
			    		\end{align}
			    		thus
			    		 $$\int_{\Omega}(u-\bar{u})gdx=0, \hbox{ for every }
			    		 g\in C_{c}^{\infty}(\Omega),$$
			    hence $u=\bar{u}$ a.e. in $\Omega$.
			    
			    This completes the proof of Proposition \ref{dual_weak}
			    and also the proof of Theorem \ref{existence}.
				\end{proof}	
				
				\section{Epilogue}\label{s5}
				\begin{remark}
			    One may note that $u$ is always a subsolution
			     of problem \eqref{main} 
			     even if the solution space is $W_0^{1,2}(\Omega)$ and not $\mathcal{X}$.
				\end{remark}
				
				The existence of a weak solution 
				of problem \eqref{main}
				 gives rise to a natural question on the multiplicity of  solutions
				 of problem \eqref{main}. 
				 One may note again that not having an energy functional is a huge drawback to come to any such conclusion. Towards this, we propose the following conjecture.
				\begin{conjecture}\label{multiplicity}
				For each $M$ in $\mathbb{N}_0$, there exists at least one distinct solution $u_M$ (having atmost $M$ modes) of problem \eqref{main}. 
				\end{conjecture}
				
				We next prove an \textcolor{blue}{auxiliary} result that can be used to pass the limit inside the Kirchhoff function.
				\begin{proposition}\label{s_to_1}
				If $(\nabla_{\mathbb{H}}f_n)$ converges weakly to $\nabla_{\mathbb{H}}f$ in $L^2(\Omega)$ and the number of modes for the sequence $(|\nabla_{\mathbb{H}} f_n|)$ is uniformly bounded, then $|\nabla_{\mathbb{H}}f_n|\to |\nabla_{\mathbb{H}}f_n|$ in $L^2(\Omega)$.
				\end{proposition}
				\begin{proof}
					We begin by observing that the sequence $(|\nabla_{\mathbb{H}}f_n|)$ having a bounded number of modes, say $M$, is a sequence which is bounded in the space $\text{BV}(\Omega)$. By the Remark \ref{BV_Compact} since $\text{BV}(\Omega)$ is compactly embedded in $L^1(\Omega)$, hence $|\nabla_{\mathbb{H}}f_n|\to g$ in $L^1(\Omega)$. Uniqueness of the limit implies that $g=|\nabla_{\mathbb{H}}f|$. Thus  there exists a subsequence, still denoted as $(|\nabla_{\mathbb{H}}f_n|)$, such that $|\nabla_{\mathbb{H}}f_n|\to|\nabla_{\mathbb{H}}f|$ in measure. Thus by the Vitali's convergence theorem, a sequence converging in measure and being \textcolor{blue}{uniformly} integrable implies that it converges strongly, i.e., $$\int_{\Omega}|\nabla_{\mathbb{H}}(f_n-f)|^2dx\to 0~\text{as}~n\to\infty.$$
				\end{proof}				
				\begin{remark}\label{zeidler}
			For a reference to monotone operators see Zeidler \cite{Zeidler}.
				\end{remark}
				We conclude the paper by presenting
				an interesting application of  Theorem \ref{existence}.
				
				\begin{example}
				Consider the following problem in the Euclidean setup:
				\begin{equation}
				\left\{\begin{aligned}\label{example_prob}
				-\Delta u+\mu\vec{U}\cdot\nabla u=&h~\text{ \textcolor{blue}{in} }\Omega\\
				u=&0~\text{ on }\partial\Omega.
			    \end{aligned}\right.
			    \end{equation}
				This is the well-known {\it diffusion-reaction equation} with an external force $h\in L^2(\Omega)$. The applicability of this equation is far-reaching and hence is important to study this equation and its associated solution(s). Here $\mathfrak{m}(\cdot)=1$ and $\nu=h/2$. The existence of at least one solution of
				problem \eqref{example_prob} can be established by invoking our
				 Theorem \ref{existence}. 
				 \end{example}
				
				\begin{remark}\label{appl1}
					{A natural generalization to this problem which the readers may consider would be the parabolic version of this problem. Although we leave this problem open, however the readers may refer to the work due to Zuo et al. \cite{zuo2} for the techniques used. }
				\end{remark}
				
		
		
		
\subsection*{Acknowledgements}
		Choudhuri thanks Prof. D.D.Repov\v{s}, University of Ljubljana and at the Institute of Mathematics, Physics and Mechanics in Ljubljana, for many useful discussions and suggestions on this work. Choudhuri also thanks the Department of Mathematics,  National Institute of Technology Rourkela, India since part of this research was carried out by the author during his visit. 
		
		\subsection*{Funding}
		D. Choudhuri was supported by the N.B.H.M., Department of Atomic Energy (DAE) India, [02011/47/2021/NBHM(R.P.)/R \& D II/2615]. O.H.Miyagaki  was partially supported by CNPQ/Brazil Proc. 303256/2022-2 and FAPESP/Brazil Proc 2022/16407-1. 
		
		\subsection*{Data Availability} Data sharing not applicable to this article as no datasets were generated or analysed during the current study.
		\subsection*{Competing Interests}
		The authors hereby state that there are no conflicts of interest regarding the presented results.

\end{document}